\documentclass[11pt,leqno]{amsart}
\usepackage[utf8]{inputenc}
\usepackage{amsfonts,amssymb}
\usepackage{graphicx,tikz}
\usepackage{tikz-cd} 
\usepackage{float}
\usepackage{centernot}
\usepackage{amsmath, amsthm, amsfonts}
\usepackage{geometry}
\usepackage{hyperref}
\usepackage[capitalise]{cleveref}
\usepackage{enumerate}
\usepackage{enumitem}
\usepackage{amsrefs}
\usepackage{nicefrac}
\usepackage[official]{eurosym}
\usepackage{graphicx}
\usepackage{nomencl}
\usepackage{amsfonts}
\usepackage{hyperref}
\usepackage{amssymb}
\usepackage{fancyhdr}
\usepackage{amscd}

\theoremstyle{definition}
\newtheorem{definition}{Definition}[section]

\theoremstyle{remark}

\theoremstyle{plain}

\theoremstyle{plain}
\newtheorem{cor}[definition]{Corollary}

\theoremstyle{plain}

\theoremstyle{plain}
\newtheorem{lemma}[definition]{Lemma}

\theoremstyle{plain}
\newtheorem{theorem}[definition]{Theorem}

\newtheorem{theoremA}[]{Theorem}

\theoremstyle{remark}

\theoremstyle{remark}
\newtheorem{notation}[definition]{Notation}

\theoremstyle{definition}

\newcommand{\D}{\mathrm{D}}

\newcommand{\Aut}{\mathrm{Aut}}

\newcommand{\Syl}{\mathrm{Syl}}

\newcommand{\norm}{\mathrel{\unlhd}}

\def \Syl {\hbox {\rm Syl}}

\title[On finite groups, arithmetic and geometric means of element orders]{On the structure of finite groups determined by the arithmetic and geometric means of element orders}

\author[V. Grazian]{Valentina Grazian}
\address{Valentina Grazian: Department of Mathematics and Applications, University of Milano - Bicocca, Via Roberto Cozzi 55, 20125 Milano, Italy} 
\email{valentina.grazian@unimib.it}
\author[C. Monetta]{Carmine Monetta}
\address{Carmine Monetta: Department of Mathematics, University of Salerno, via Giovanni Paolo II 132, 84084 Fisciano (SA), Italy}
\email{cmonetta@unisa.it}
\author[M. Noce]{Marialaura Noce}
\address{Marialaura Noce: Department of Mathematics, University of Salerno, via Giovanni Paolo II 132, 84084 Fisciano (SA), Italy}
\email{mnoce@unisa.it}


\usepackage{graphicx}
\begin{document}

\maketitle
\begin{abstract}
In this paper we consider two functions related to the arithmetic and geometric means of element orders of a finite group, showing that certain lower bounds on such functions strongly affect the group structure. In particular, for every prime $p$, we prove a sufficient condition for a finite group to be $p$-nilpotent, that is, a group whose elements of $p'$-order form a normal subgroup. Moreover, we characterize finite cyclic groups with prescribed number of prime divisors.
\end{abstract}

\section{Introduction}
Let $p$ be  a prime number. A finite group $G$ is said to be $p$-nilpotent if it admits a normal $p$-complement, that is, a normal subgroup $H$ of $G$ of $p'$-order and $p$-power index in $G$. The search for $p$-nilpotency criteria has been the subject of various works. In the last century the problem was approached looking at normalizers of $p$-subgroups of $G$, as in the famous Frobenius normal $p$-complement theorem, \cite[10.3.2]{ROB}. However, in recent years, the study of $p$-nilpotency took a different direction, focusing on properties of elements' orders of the group $G$. For instance in \cite{BS} and \cite{CGM}, some $p$-nilpotency criteria are obtained investigating the order of the product of element of coprime orders.  

The aim of this work is to characterize $p$-nilpotency in terms of functions related to the arithmetic and geometric means of the orders of all elements of $G$.

For a finite group $G$, let $\psi(G)$ and $\rho(G)$ denote the sum and the product of element orders of $G$, respectively. For any positive integer $n$ and any group $G$ of order $n$, if we denote by $C_n$ the cyclic group of order $n$, then $\psi(G) \leq \psi(C_n)$  and $\rho(G) \leq \rho(C_n)$, and the equalities  hold if and only if $G$ is cyclic (see \cite{AJI, GP}). When $G$ is a non-cyclic group of order $n$, these upper bounds can be sharpened using other functions depending both on $n$ and on the least prime number dividing $n$ (for example see \cite{HML1, DMN}). As a result one can obtain criteria or necessary conditions for a finite group to be nilpotent, soluble or supersoluble as in \cite{HML3, HML4}: we refer to \cite{HML2} for a survey on this topic. Recently, some functions related to $\psi(G)$ and $\rho(G)$ have been considered, namely the functions $$\psi''(G)=\displaystyle \frac{\psi(G)}{|G|^2} \quad \text{ and  } \quad l(G) = \displaystyle \frac{\rho(G)^{1/|G|}}{|G|}.$$

We point out that $\psi''(G) \cdot|G|$ and $l(G) \cdot|G|$ are the arithmetic and geometric means of the element orders of $G$, respectively. In particular $\psi''(G) \geq l(G)$. For the convenience of the reader, we report some approximated values of $\psi''(G)$ and $l(G)$ for certain finite groups, computed with the software GAP: 

\begin{table}[H]
    \begin{tabular}{| c | c | c | c | c | c |}
         \hline
         $G$ & $C_2 \times C_2$ & $Q_8$ & $S_3$ & $A_4$ & $A_5$\\
         \hline
         $\psi''(G)$ & 0.437 & 0.422 & 0.361 & 0.215 & 0.059\\
         %\hline
        $l(G)$ & 0.420 & 0.385 & 0.339 & 0.206 & 0.054\\
        \hline
    \end{tabular}
    \caption{Some values of $\psi''(G)$ and $l(G)$.}
\end{table}

where $Q_8$, $S_n$ and $A_n$ denote the quaternion group of order $8$, the symmetric group and the alternating group of degree $n$, respectively.

One of the reasons to focus on the function $l$ instead of the function $\rho$ is that only the first one is multiplicative under certain assumptions. More precisely, If $G=H \times K$ is the direct product of two finite groups of coprime orders, then $\rho(G)=\rho(H)^{|K|} \rho(K)^{|H|}$ (see \cite{DMN}), but $l(G)=l(H)l(K)$ (see Lemma \ref{prop}(2)). As a consequence, in the study of the function $l$ one can apply a variety of techniques that fail for the function $\rho$.

In \cite{Tarna1} and \cite{AK}, some sufficient conditions for a finite group $G$ to be cyclic, abelian, nilpotent, supersoluble, and soluble have been proved using the values of the functions $\psi''(G)$ and $l(G)$.

\begin{theorem}{\cite[Theorem 1.1]{Tarna1} and \cite[Theorem 1.1 and 1.2]{AK}}\label{Thm1.1}
Let $G$ be a finite group and let $f \in \{\psi'', l\}$.
\begin{itemize}
\item[(a)] If $f(G) > f(C_2 \times C_2)$, then $G$ is cyclic.
\item[(b)] If $f(G) > f(Q_8)$, then $G$ is abelian.
\item[(c)] If $f(G) > f(S_3)$, then $G$ is nilpotent.
\item[(d)] If $f(G) > f(A_4)$, then $G$ is supersoluble.
\item[(e)] If $f(G) > f(A_5)$, then $G$ is soluble.
\end{itemize}
\end{theorem}

In this paper we will focus our attention on the $p$-nilpotency of a finite group with respect to the function $l(G)$. To begin with, assume $p=2$. Then a natural lower bound to consider is $l(A_4)$. Indeed, recalling that supersoluble groups are $2$-nilpotent (see for example \cite[5.4.9]{ROB}), the following is an immediate consequence of \cite[Theorem 1.1]{AK}.

\begin{theoremA}\label{p=2}
 Let $G$ be a finite group. If $l(G) > l(A_4)$ then $G$ is $2$-nilpotent. 
\end{theoremA}

Since $A_4$ is not $2$-nilpotent, we point out that the bound above is the best possible.

Now, let $p$ be an odd prime, and consider the value
\[ l(D_{2p})= 2^{-\frac{1}{2}} \cdot  p^{-\frac{p+1}{2p}},\]
where $D_{2n}$ denotes the dihedral group of order $2n$.
Note that $l(D_{2p})$ is a decreasing function of $p$. 
\begin{table}[H]
    \centering
    \begin{tabular}{| c | c | c | c | c | c || c |}
    \hline
         $p$ & $3$ & $5$ & $7$ & 11 & 13 & 173\\
         \hline
        $l(D_{2p})$ & 0.340 & 0.270 & 0.233 & 0.191 & 0.178 & 0.055 \\
        \hline
    \end{tabular}
    \label{tab:my_label}
\smallskip

    \caption{Some approximated values of $l(D_{2p})$.}
\end{table}

If $G$ is a finite group and $p$ is a prime, we denote by $O_p(G)$ and $O_{p'}(G)$ the largest normal $p$-subgroup and $p'$-subgroup of $G$, respectively. Our main result is the following:
\begin{theoremA}\label{main} Let $G$ be a finite group and let $p$ be an odd prime dividing the order of $G$. Suppose \begin{equation*}l(G) \geq l(D_{2p}). \end{equation*} Then either
$G \cong D_{2p}$ or
$l(G) > l(D_{2p})$ and $G = O_p(G) \times O_{p'}(G)$ with $O_p(G)$ cyclic. In particular, if $l(G) > l(D_{2p})$, then $G$ is $p$-nilpotent.
\end{theoremA}

We point out that if $l(G)=l(\D_{2n})$, where $n$ is an odd integer not necessarily prime, then in general it is not true that $G \cong D_{2n}$. For instance, $l(S_3 \times C_3) = l(\D_{18})$.
Moreover, the lower bound of Theorem \ref{main} is the best possible for every odd prime $p$, as the group $D_{2p}$ is not $p$-nilpotent. We remark that in \cite[Main Theorem]{AKpnil}, the authors proved that a similar result holds with respect to the function $\psi''(G)$, that is, if $p$ is an odd prime and $\psi''(G) > \psi''(\D_{2p}) = \frac{p^2+p+1}{4p^2}$ then $G = O_p(G) \times O_{p'}(G)$ with $O_p(G)$ cyclic. 
Furthermore, in the recent work \cite[Theorem 1.2]{BAKJ}, it is proved that if $p>5$, $\psi''(G) = \psi''(D_{2p})$ and $G$ is not $p$-nilpotent then $G \cong D_{2p}$.

In \cite{AKpnil}, the authors also show that if $\psi''(G) > \psi''(\D_{2p})$ then $G$ is supersoluble with second derived subgroup contined in the center of $G$, namely, $[G',G'] \leq Z(G)$. However, this does not hold in general when $l(G) > l(\D_{2p})$, representing an important difference among the two functions $\psi''$ and $l$. Indeed, there are infinite families of non-supersoluble finite groups satisfying the hypothesis of Theorem \ref{main}:  
\begin{itemize}
\item if $p > 13$ then $l(C_p \times A_4) > l(D_{2p})$ and $C_p \times A_4$ is not supersoluble;
%group with  $l(G)  = l(C_p)l(A_4) = p^{-1/p} \cdot (2)^{-7/4}(3)^{-1/3} > l(D_{2p})$.
\item if $p > 173$ then $l(C_p \times A_5) > l(D_{2p})$ and $C_p \times A_5$ is not soluble.
%$l(C_p)l(A_5) = p^{-1/p} \cdot (2)^{-7/4}(3)^{-2/3}(5)^{-3/5} > l(D_{2p})$.
\end{itemize}

However, for small primes $p$, we can actually prove that if $f \in \{\psi'', l\}$ and $G$ is a group whose order is divisible by $p$ satisfying $f(G)> f(D_{2p})$, then $G$ is supersoluble and, in some cases, even cyclic or nilpotent: 

\begin{theoremA}\label{p=3} Let $G$ be a finite group whose order is divisible by the odd prime $p$. Suppose  $f \in \{\psi'', l\}$ and $f(G) > f(D_{2p})$.
\begin{enumerate}
\item If $p=3$ then $G$ is cyclic.
\item If $p \leq 5$ then $G$ is nilpotent.
\item If $p \leq 13$ then $G$ is supersoluble.
\end{enumerate}
\end{theoremA}

Note that 
if $p > 3$ then $f(C_p \times Q_8)> f(D_{2p})$ and $C_p \times Q_8$ is not cyclic and
if $p > 5$ then $f(C_p \times S_3) > f(D_{2p})$ and $C_p \times S_3$ is not nilpotent,
showing that the choice of primes in Theorem \ref{p=3} is sharp.
Moreover, thanks to Theorem \ref{p=3}, we are able to establish that more groups than the ones characterized in \cite[Theorem 1.1]{Tarna1} and \cite[Theorem 1.2]{AK} are cyclic or nilpotent. For instance,
 $l(D_6) < l(C_{12}) < l(C_2 \times C_2)$,  $\psi''(D_6) < \psi''(C_{180}) < \psi''(C_2 \times C_2)$
and $f(D_{10}) < f(C_{5} \times Q_8) < f(S_3)$.

\medskip
Theorem \ref{main} also allows us to characterize nilpotency in some special cases, described by the following corollary.%corollaries. 

\begin{cor}\label{cor:two.primes}
Let $G$ be a finite group such that $|G|=p^aq^b$ for distinct primes  $p$ and $q$, with $p\neq 2$. If $l(G) > l(D_{2p})$ then $G$ is nilpotent.
\end{cor}

\begin{proof}
By Theorem \ref{main} we get $G = O_p(G) \times O_{p'}(G)$. From $|G|=p^aq^b$, we deduce that $O_{p'}(G) = O_{q}(G) \in \Syl_q(G)$. Hence all Sylow subgroups of $G$ are normal in $G$ and $G$ is nilpotent. 
\end{proof}

Finally, Theorem \ref{main} allows us to characterize cyclic groups whose order has few prime divisors.

\begin{cor}\label{cor:odd}
Let $G$ be a finite group of odd order and let $p$ be the smallest prime divisor of $|G|$. If $l(G) > l(D_{2p})$ then $G$ is cyclic. Moreover, if the number of distinct primes dividing $|G|$ is at most $\frac{p+1}{2}$, then $G$ is cyclic if and only if $l(G) > l(D_{2p})$.
\end{cor}

In general, if $p\geq 3$ is the smallest prime divisor of the order of a group $G$, but there are more than $\frac{p+1}{2}$ distinct primes dividing  $|G|$, then the property $l(G)>l(D_{2p})$ is not a necessary condition for $G$ to be cyclic. For instance, the cyclic group of order $315 = 3^2*5*7 $ satisfies $l(C_{315}) \approx 0.337 < l(D_6)$.

\section{Properties of the functions  $\psi''(G)$ and $l(G)$}
In this section we analyse some properties of the functions   $\psi''(G)$ and $l(G)$.
When $G$ is a nontrivial group, then the following is clear.

\begin{lemma}\label{lem:limit}
For any nontrivial finite group $G$ we have $0 < l(G) \leq \psi''(G) < 1$.
\end{lemma}

\begin{lemma}\label{lem:strict.quotient}
Let $G$ be a finite group, let $N$ be a non-trivial proper normal subgroup of $G$ and let $f \in \{\psi'',l\}$. 
Then $$f(G) < f(G/N).$$
\end{lemma}
\begin{proof}
If $f = \psi''$, this is \cite[property (1)]{Tarna1} for the inequality and \cite[Lemma 3.2]{BAKJ} for the strict one. Suppose $f=l$.
Let $a_1, \dots, a_m \in G$ be such that 
$G= a_1N \cup a_2N \cup \dots \cup a_mN$ and
$a_iN \neq a_jN$ whenever $i\neq j$. 
So 
\[
\rho(G) = \prod_{i =1}^m\prod_{n \in N} o(a_in),
\]
where $o(a_in)$ denotes the order of the element $a_in$.
For every $1 \leq i \leq m$, let $t_i$ be the smallest positive integer such that $a_i^{t_i} \in N$. Then for every $n\in N$ we have $(a_in)^{t_i} \in N$ and so $o(a_in)$ divides $t_i|N|$. Also, we can assume that $a_1=1$ and so $t_1 = 1$ and for $n=1$ we get $o(a_1n) = 1 < 1\cdot |N|$, as $N$ is non-trivial. Thus there is at least one couple $(a_i,n)$ with $o(a_in) < t_i|N|$. Therefore 
 $$ \rho(G) < (t_1|N|)^{|N|} \cdot (t_2|N|)^{|N|} \cdots (t_m|N|)^{|N|} = \rho(G/N)^{|N|}|N|^{m|N|} = \rho(G/N)^{|N|}|N|^{|G|}.$$
Using the definition of $l(G)$, we get 
\[ l(G) = \frac{\rho(G)^{1/|G|}}{|G|} < \frac{\rho(G/N)^{|N|/|G|}|N|}{|G|} =l(G/N).\]
%Finally, suppose for a contradiction that $l(G)=l(G/N)$, that is, $\rho(G) = \rho(G/N)^{|N|}|N|^{|G|}$ and so $o(a_in) = t_i|N|$ for every $a_i$ and every $n\in N$. In particular, taking $n=1$, since $N\neq 1$ we deduce that $1 \neq a_i^{t_i} \in N$ and that every non-trivial element of $N$ has order exactly $|N|$. Therefore $N$ must have order $p$, for some prime number $p$, and it is the only subgroup of $G$ of such order. Also, $G$ is a $p$-group, otherwise we can find an element $a_i$ of order $q\neq p$. By assumption $p > 2$ so by \cite[Theorem 5.4.10.ii]{GOR} we deduce that $G$ is a cyclic $p$-group. Now we can apply Lemma \ref{cyclic.group} to compute the exact values of $l(G)$ and $l(G/N)$ and since $|G| > |G/N|$ we conclude that $l(G) < l(G/N)$. 
\end{proof}

\begin{lemma}\label{prop} Let $G$ be a finite group and let $f \in \{\psi'',l\}$. Then
\begin{enumerate}
\item if $G$ is a cyclic group and $K$ is a proper subgroup of $G$, then $f(G) < f(K)$;
\item if $H$ is a finite group and $(|G|,|H|)=1$ then $f(G \times H) = f(G)\cdot f(H)$;
\item if $f(G) > \alpha$ for some positive real number $\alpha$, then there exists and element $x\in G$ such that $[ G \colon \langle x \rangle] < \alpha^{-1}$.
\end{enumerate}
\end{lemma}

\begin{proof}
To prove (1), let $t$ be the index of $K$ in $G$. As $G$ is cyclic, there exists a subgroup $H$ of $G$ of order $t > 1$. Then $K$ is isomorphic to $G/H$ and $f(G) < f(G/H)=f(K)$ by Lemma \ref{lem:strict.quotient}.
Statement (2) is  \cite[Lemma 1.2(2)]{Tarna1} and  \cite[Lemma 2.7]{AK}. Finally, part (3) for $\psi''$ is \cite[Lemma 2.1]{AKpnil}. Now, if $l(G) > \alpha$ then $ \psi''(G) \geq l(G) > \alpha$ and the result follows.  
\end{proof}

If $p$ is a prime and $P$ is a finite cyclic $p$-group, then we can compute the exact values of $\psi''(P)$ and $l(P)$:

\begin{lemma}\label{cyclic.group}
Let $p$ be a prime and let $P$ be a non-trivial cyclic $p$-group of order $p^n$. Then 
\[\psi''(P) =  \frac{p^{2n+1} + 1}{p^{2n +1}+  p^{2n}} \quad \text{ and } \quad l(P) = p^{-\frac{p^n - 1}{p^n(p-1)}}.\]
In particular $p^{-\frac{1}{p-1}} \leq l(P) \leq p^{-\frac{1}{p}}$ and, if $p$ is odd, then $l(P) > p^{-\frac{p+1}{2p}} > l(D_{2p})$.
\end{lemma}

\begin{proof}
The formulas for $\psi''(P) $ and $l(P)$ follow from \cite[Lemma 2.9(1)]{HML1} and \cite[Lemma 2.10]{AK}. As for the last statement, note that, if we fix the prime $p$, then the function $p^{-\frac{p^n - 1}{p^n(p-1)}}$ is decreasing in $n$. Moreover, if $n=1$ then $l(P) =  p^{-\frac{1}{p}}$ and $\lim_{n\rightarrow +\infty}p^{-\frac{p^n - 1}{p^n(p-1)}} = p^{-\frac{1}{p-1}}$. Hence for every cyclic $p$-group $P$ we have $ p^{-\frac{1}{p-1}} \leq  l(P) \leq p^{-\frac{1}{p}}$. Finally, if $p$ is odd, then $2p < p^2-1$ and we deduce that 
$$l(P) \geq p^{-\frac{1}{p-1}} > p^{-\frac{p+1}{2p}} > 2^{-\frac{1}{2}}p^{-\frac{p+1}{2p}} = l(D_{2p}).$$
\end{proof}

%\begin{proof}
%The By  we obtain
%\[ \psi''(P) = \frac{\psi(P)}{p^{2n}}\]
 %we get
%\[ l(P) = \frac{\rho(P)^{1/p^n}}{p^n} = %\frac{(p^{np^n - \frac{p^n - 1}{p-1}})^{1/p^n}}{p^n} %= p^{-\frac{p^n - 1}{p^n(p-1)}}.\]\end{proof}

Going further, we obtain a similar lower bound for $l(G)$ when $G$ is a cyclic group whose order is odd and divisible by a bounded number of primes. 

\begin{lemma}\label{cyclic.few.divisors}
Let $G$ be a finite cyclic group of odd order and suppose $p\geq 3$ is the smallest prime dividing $|G|$. If the number of prime divisors of $|G|$ is at most $\frac{p+1}{2}$ then $l(G) > l(D_{2p})$.
\end{lemma}

\begin{proof}
let $p=p_1 < p_2 < \dots < p_r$ be the distinct prime divisors of $|G|$ and let $P_i$ be a Sylow $p_i$-subgroup of $G$ for every $1\leq i \leq r$. Then by Lemma \ref{prop}(2) we have
\[l(G) = \prod_{i=1}^{r}l(P_i). \]
Now, for every $i \geq 2$ we have $l(P_i) > p_i^{-\frac{1}{p_i-1}}$ by Lemma \ref{cyclic.group}. 
Since the function $x^{-\frac{1}{x-1}}$ is increasing for $x> 1$, for every $i \geq 2$ we deduce $l(P_i) > p^{-\frac{1}{p-1}}$. Hence
\[l(G) = \prod_{i=1}^{r}l(P_i) > p^{-\frac{r}{p-1}} \]
and  $p^{-\frac{r}{p-1}} > l(D_{2p}) = 2^{-\frac{1}{2}} \cdot  p^{-\frac{p+1}{2p}}$ if and only if $p^{\frac{p^2 -2rp -1}{2p(p-1)}} > 2^{-\frac{1}{2}}. $

Suppose $r\leq \frac{p+1}{2}$. Then 
\[ p^{\displaystyle \frac{p^2 -2rp -1}{2p(p-1)}} > p^{\displaystyle \frac{-p -1}{2p(p-1)}}.\]
Now, $p^{\frac{-p -1}{2p(p-1)}}$ is an increasing function of $p$ and if $p=5$ then  $p^{\frac{-p -1}{2p(p-1)}} = 5^{\frac{-6}{40}} >  2^{-\frac{1}{2}}.$ This proves that for $p\geq 5$ we have $l(G) > l(D_{2p})$. Finally, suppose $p=3$ and $|G|$ is divided by two primes, say $3$ and $q$. Then
\[ l(G) \geq 3^{-\frac{1}{2}} \cdot q^{-\frac{1}{q-1}} \geq 3^{-\frac{1}{2}} \cdot 5^{-\frac{1}{4}} \approx 0.386 > 0.340 \approx l(D_{6}). \]
Hence the statement is also true for $p=3$.
\end{proof}

We conclude the section with some structural results involving the functions $\psi''$ and $l$.

\begin{lemma}\label{cyclic.sylow}
Let $G$ be a finite group, let $p$ be a prime and let $f \in \{\psi'',l\}$. If $f(G) > p^{-1}$ then $G$ has a unique (normal) Sylow $p$-subgroup $P$ and $P$ is cyclic.
\end{lemma}

\begin{proof}
By Lemma \ref{prop} (3), there exists an element $x\in G$ such that
\[ [G \colon \langle x \rangle] < p \]
 and so $p$ does not divide $[G \colon \langle x \rangle]$. Let $P \in \Syl_p(\langle x \rangle)$. Then $P$ is cyclic and $P \in \Syl_p(G)$. 
Since $\langle x \rangle \leq N_G(P)$, it follows that $[G \colon N_G(P) ] < p$, which implies $[G \colon  N_G(P) ]=1$.
%Set $K:= \rm{core}_G(\langle x \rangle)$. By Lemma \ref{Lucchini} we have $[\langle x \rangle \colon K] < [G \colon \langle x \rangle] < p$, and so we deduce that $P \leq K$. In particular $P$ is cyclic. Now $K$ is cyclic and normal in $G$ and $P$ is characteristic in $K$. Therefore we conclude that $P\norm G$ and so $P$ is the unique Sylow $p$-subgroup of $G$.
\end{proof}

The next lemma gives examples of finite groups $G$ satisfying $l(G) > p^{-1}$ that are not $p$-nilpotent. First we need some notation:

\begin{notation} Let $p$ be a prime. For $x \in \mathbb{R} \setminus \{0\}$ we set
\[\theta_p(x):=p^{\frac{-1+p-px}{px}}.\]
\end{notation}

We point out that the real function $\theta_p(x)$ is decreasing for every prime $p$.

\begin{lemma}\label{lem:semidirect}
Let $p$ be an odd prime and let $G$ be a finite group such that $G = PH$ with $P\in \Syl_p(G), P\cong C_p, P\norm G, C_G(P) = P$, $|H| > 1$ and $P \cap H = 1$. Then
\begin{enumerate}
\item $l(G) = \theta_p(|H|)\cdot l(H).$
\item if $|H| = q$ for some prime $q$, then  
either $q=2$ and $G \cong D_{2p}$ or $q\geq 3$ and  $p^{-1} < l(G) < l(D_{2p})$.
\end{enumerate}
\end{lemma}

\begin{proof} Note that $P$ is the unique Sylow $p$-subgroup of $G$ and so $(|P|,|H|)=1$.
By \cite[Lemma 2.6]{AK} we get 
\[ \rho(G) = \rho(P)^{|C_H(P)|}\cdot\rho(H)^{|P|} = p^{p-1} \cdot \rho(H)^p\]
Hence
\[ l(G) =  \frac{\rho(G)^{\frac{1}{p|H|}}}{p|H|} = \frac{(p^{p-1} \cdot \rho(H)^p)^{\frac{1}{p|H|}}}{p|H|} = \frac{p^\frac{p-1}{p|H|}}{p} \cdot \frac{\rho(H)^{\frac{1}{|H|}}}{|H|} = p^{-\frac{p|H| - p + 1}{p|H|}}\cdot l(H) = \theta_p(|H|)\cdot l(H).\]

Now suppose $|H|=q$ for some prime $q$. Hence $H\cong C_q$ and $l(H) = q^{-\frac{1}{q}}$ by Lemma \ref{cyclic.group}.
If $q=2$ then $G \cong D_{2p}.$
%and so $l(G) = l(D_{2p}) > p^{-1}$. 
Now let $q \geq 3$ and, aiming for a contradiction, suppose $l(G) \geq l(D_{2p})$. By part (1) we get
\[p^{\frac{-pq +p - 1}{pq}} \cdot q^{-\frac{1}{q}} = l(G) \geq l(D_{2p}) = 2^{-\frac{1}{2}}p^{-\frac{p+1}{2p}} \]
that corresponds to
\[ 2^{pq} \geq q^{2p} p^{(p-1)(q-2)}.\]

Note that $H \cong G/P = N_G(P)/C_G(P)$ is isomorphic to a subgroup of $\Aut(C_p) \cong C_{p-1}$. Hence $q$ divides $p-1$ and from $q\geq 3$ we deduce that $p > 2q$. Thus
\[ 2^{pq} \geq q^{2p} p^{(p-1)(q-2)} > q^{2p} \cdot (2q)^{(p-1)(q-2)} = q^{2p + (p-1)(q-2)} \cdot 2^{(p-1)(q-2)}\]
and so
\[ 2^{pq - (p-1)(q-2)} > q^{2p + (p-1)(q-2)} > 2^{2p + (p-1)(q-2)}.\] 

Now this is true if and only if
\[\begin{aligned} pq - (p-1)(q-2) &> 2p + (p-1)(q-2) 
\\\Leftrightarrow pq - pq +2p +q - 2 &> 2p + pq -2p -q + 2
\\\Leftrightarrow pq &< 2q + 2p - 4 < 3p - 4 < 3p  \end{aligned}\]
where in last line we used again that $2q < p$, obtaining a contradiction. Therefore $l(G) < l(D_{2p})$. 

It remains to show that $l(G) > p^{-1}$, that is 

\[p^{\frac{-pq +p - 1}{pq}} \cdot q^{-\frac{1}{q}} > p^{-1} \Leftrightarrow p^p > p \cdot q^p.\]

Note that $2^p > p$ and $q\leq \frac{p-1}{2}$, so 
\[p^p > \frac{p}{2^p} p^p = p \cdot \left( \frac{p}{2} \right)^p >p \cdot q^p.\]
 This completes the proof.
\end{proof}

\section{Proof of the main results }
In this final section we prove our main results: Theorem \ref{main}, Theorem \ref{p=3} and Corollary \ref{cor:odd}. We start with some useful lemmas. We recall that for a finite group $G$, let $O_{p'}(G)$ denote the largest normal $p'$-subgroup of $G$ and let $\Phi(G)$ denote the Frattini subgroup of $G$.

%\begin{lemma}{\cite[9.1.2]{ROB}}{(Schur-Zassenhaus)}\label{SZ}
%If $G$ is a finite group and $N \norm G$ is such that $(|N|,[G \colon N]) = 1$, then there exists a %subgroup $H$ of $G$ of order $[G \colon N]$ (so in particular $G=NH$ and $H \cap N = 1$).
%\end{lemma}

\begin{lemma}\label{quotients}
Let $p$ be a prime and let $G$ be a finite group. If $G/\Phi(G)$ or $G/O_{p'}(G)$ is $p$-nilpotent, then $G$ is $p$-nilpotent.
\end{lemma}

\begin{proof}
If $G/\Phi(G)$ is $p$-nilpotent, then $G$ is $p$-nilpotent by \cite[9.3.4]{ROB}. Suppose $G/O_{p'}(G)$ is $p$-nilpotent. Then there exists a normal subgroup $H$ of $G$ containing $O_{p'}(G)$ such that $H$ has $p$-power index in $G$ and $([H \colon O_{p'}(G)],p)=1$. Since $O_{p'}(G)$ has $p'$-order, we deduce that $H$ has $p'$-order and so $H$ is a normal $p$-complement for $G$, that is, $G$ is $p$-nilpotent.
\end{proof}

\begin{lemma}{\cite[9.1.1 (i), 5.2.13 (ii)]{ROB}}\label{lem:Op-Frat}
If $G$ is a finite group and $N \norm G$ then $O_{p'}(N) \leq O_{p'}(G)$ and  $\Phi(N) \leq \Phi(G)$.
\end{lemma}

%\begin{lemma}{\cite[5.4.9]{ROB}}\label{supersol}
%If $G$ is a finite supersoluble group then the elements of odd order of $G$ form a characteristic subgroup of $G$. In particular $G$ is $2$-nilpotent.
%\end{lemma}

\begin{lemma}\label{p.smaller}
If $p \geq 3$ is a prime, then $l(D_{2p}) > p^{-1}$.
\end{lemma}

\begin{proof}
We have
\[ l(D_{2p}) = {2}^{-\frac{1}{2}} \cdot  p^{-\frac{p+1}{2p}}  > p^{-1}  
\Longleftrightarrow  2^{-p} \cdot p^{-p-1} > p^{-2p}  \Longleftrightarrow p^{p-1} > 2^p \Longleftrightarrow p\geq 3.\]
\end{proof}

Now, we are in a position to prove Theorem \ref{main}.

\begin{proof}[Proof of Theorem \ref{main}]
Suppose $p$ is an odd prime dividing the order of $G$.
We start proving that if $l(G) \geq l(D_{2p})$ then either $G \cong D_{2p}$ or $l(G) > l(D_{2p})$ and $G$ is $p$-nilpotent. Applying Lemma  \ref{p.smaller} and then Lemma \ref{cyclic.sylow} we immediately get that $P:=O_p(G)$ is a Sylow $p$-subgroup of $G$ and it is cyclic.
Aiming for a contradiction, suppose there exists a finite group $G$ such that $l(G) \geq l(D_{2p})$ but either $l(G) = l(D_{2p})$ and $G$ is not isomorphic to $D_{2p}$ or $l(G) > l(D_{2p})$ and $G$ is not $p$-nilpotent.
%\textcolor{blue}{either $l(G) = l(D_{2p})$ and $G$ is not isomorphic to $D_{2p}$ or $l(G) > l(D_{2p})$ and $G$ is not $p$-nilpotent}.
We can take such a $G$ minimal with respect to the order. In particular, every finite group $H$ with $|H| < |G|$ and $l(H)> l(D_{2p})$ is $p$-nilpotent.
\begin{description}
\item[Claim 1: $ \bf O_{p'}(G) = \Phi(G) = 1$] Let $N \in \{O_{p'}(G), \Phi(G)\}$ and, aiming for a contradiction, suppose $N \neq 1$. 
By Lemma \ref{lem:strict.quotient} we have $l(G/N) > l(G) \geq l(D_{2p})$. By minimality of $G$ we deduce that $G/N$ is $p$-nilpotent and so $G$ is $p$-nilpotent by Lemma \ref{quotients}, a contradiction. Thus $N=1$, as wanted.

\item[Claim 2: $ \bf P \cong C_p$]
 Since $P$ is normal in $G$, we deduce that $\Phi(P) \leq \Phi(G)$ by Lemma \ref{lem:Op-Frat}. Now, $\Phi(G) = 1$ by Claim 1, so $\Phi(P) = 1$ and $P \cong C_p$.

\item[Claim 3: $ \bf C_G(P) = P$] Note that $P\leq C_G(P)$, as $P$ is abelian, so $P \in \Syl_p(C_G(P))$. Since $P\norm C_G(P)$, the Schur-Zassenhaus Theorem \cite[9.1.2]{ROB} gives $C_G(P) = PH$ for some subgroup $H$ of $G$ of $p'$-order. Note that $H \norm C_G(P)$, as $P \leq Z(C_G(P))$. Since $P$ is characteristic in $G$, so is $C_G(P)$ and we deduce that $H \leq O_{p'}(G)=1$ by Claim 1.  Therefore $P=C_G(P)$.

\item[Claim 4: $\bf G = PQ$ where $|Q| = q$ for some prime $\bf q \neq p$] Since $P \norm G$, the Schur-Zassenhaus Theorem \cite[9.1.2]{ROB} implies that there exists a subgroup $H$ of $G$ such that $G=PH$ and $P \cap H =1$.  Note that Claim 3 gives $H \cong G/P = N_G(P)/C_G(P)$, and so $H$ is isomorphic to a subgroup of $\Aut(C_p) \cong C_{p-1}$. Thus $H$ is cyclic. If $H=1$ then $G=P$ is a cyclic $p$-group, so $p$-nilpotent, and by Lemma \ref{cyclic.group} we get $l(G) > l(D_{2p})$, a contradiction. Therefore $H$ is a non-trivial cyclic group of $p'$-order. 
Let $q$ be a prime divisor of $|H|$ and let $Q$ be the unique subgroup of $H$ of order $q$. Set $G_1=PQ \leq G$. Note that $C_{G_1}(P) = P$, so $N_{G_1}(P)/C_{G_1}(P) = G_1/P \cong Q$ is a $p'$-group. By Frobenius normal $p$-complement theorem (\cite[10.3.2]{ROB}) we deduce that $G_1$ is not $p$-nilpotent. 
By Lemma \ref{lem:semidirect}(1) applied first to $G=PH$ and then to $G_1=PQ$ we get  
$l(G) = \theta_p(|H|)l(H)$ and $l(G_1)=\theta_p(|Q|)l(Q)$. We want to prove that $Q=H$. Suppose for a contradiction that $Q < H$. Since $H$ is cyclic, by Lemma~\ref{prop}(1) we get $l(Q) > l(H)$. Using the fact that the function $\theta_p(x)$ is decreasing we obtain
\[
l(G_1)=\theta_p(|Q|)l(Q) > \theta_p(|Q|)l(H) \geq \theta_p(|H|)l(H) = l(G) \geq  l(D_{2p}).
\]
So $l(G_1) > l(D_{2p})$ and the minimal choice of $G$ implies that $G_1$ is $p$-nilpotent, a contradiction. Thus $Q=H$ and $G=PQ$.
 
\end{description}

Claims 1 to 4 show that $G$ satisfies the assumptions of part (2) of Lemma \ref{lem:semidirect} and so we deduce that either $G \cong D_{2p}$ or $l(G) < l(D_{2p})$, a contradiction.

We showed that if $G$ is a finite group satisfying $l(G) \geq l(D_{2p})$ then either $G \cong D_{2p}$ or $l(G) > l(D_{2p})$ and $G$ is $p$-nilpotent.

Finally, if $l(G) > l(D_{2p})$, then  the group $O_{p'}(G)$ is a normal $p$-complement and we have already shown that $O_p(G)\in \Syl_p(G)$, so $G = O_p(G) \times O_{p'}(G).$
\end{proof}

Next, we show how Theorem \ref{main} together with \cite[Main Theorem]{AKpnil} and direct computations imply Theorem \ref{p=3}.

\begin{proof}[Proof of Theorem \ref{p=3}]
Suppose $f(G) > f(D_{2p})$.
By Theorem \ref{main} and \cite[Main Theorem]{AKpnil} we have $G = O_p(G) \times O_{p'}(G)$ with $O_p(G)$ cyclic. Since $p$ divides the order of $G$, we also have $O_p(G) \neq 1$ and so by Lemma \ref{prop}(1)  we get $f(O_p(G)) \leq f(C_p)$. Using Lemma \ref{prop}(2) we deduce that
\[ f(C_p)f(O_{p'}(G)) \geq f(O_p(G))f(O_{p'}(G)) = f(G) > f(D_{2p}) \]
and so by Lemma \ref{cyclic.group} we have

\[ \psi''(O_{p'}(G))  >  \frac{\psi''(D_{2p})}{\psi''(C_p)} = \frac{\psi''(D_{2p})}{\frac{p^3 +1}{p^3 + p^2}} = \frac{\frac{p^2+p+1}{4p^2}}{\frac{p^3 +1}{p^3 + p^2}} = \frac{p^2 + p +1}{4(p^2 - p +1)} \]

and 

\[ l(O_{p'}(G))  >  \frac{l(D_{2p})}{l(C_p)} = \frac{l(D_{2p})}{p^{-\frac{1}{p}}} = 2^{-{1/2}}p^{\frac{1-p}{2p}}. \]

The functions $\frac{p^2 + p +1}{4(p^2 - p +1)}$ and  $2^{-{1/2}}p^{\frac{1-p}{2p}}$ are decreasing for every prime $p\geq 3$.

Suppose $p=3$. Then we get
\[\psi''(O_3'(G)) > \frac{13}{28} > \psi''(C_2 \times C_2)  \quad \text{ and } \quad 
l(O_{3'}(G))  >  2^{-{1/2}}3^{-{1/3}} > l(C_2 \times C_2) \]
 and in both cases by Theorem \ref{Thm1.1} we deduce that the group $O_{3'}(G)$ is cyclic. Thus $G$ is the direct product of two cyclic groups of coprime order and is therefore cyclic.

 Suppose $p=5$. Then 
 \[ \psi''(O_5'(G)) > \frac{31}{84} > \psi''(S_3)  \quad \text{ and } \quad  l(O_{5'}(G))  >  2^{-{1/2}}5^{-{2/5}} > l(S_3) \]
 and again by Theorem \ref{Thm1.1} we conclude that $O_{5'}(G)$ is nilpotent and so $G$ is nilpotent. 

Now, by Theorem \cite[Main theorem]{AKpnil}, if $\psi''(G) > \psi''(D_{2p})$ then $G$ is supersoluble for any prime $p$. As for the function $l$, suppose $p\leq 13$. Then
\[ l(O_{p'}(G))  > 2^{-{1/2}}p^{\frac{1-p}{2p}} >  2^{-{1/2}}(13)^{-{12/26}} > l(A_4) \]
 and by \cite[Theorem 1.1(b)]{AK} we conclude that $O_{p'}(G)$ is supersoluble and so $G$ is supersoluble. 
 \end{proof}

We conclude this section proving Corollary \ref{cor:odd}.

\begin{proof}[Proof of Corollary \ref{cor:odd}]
Note that the function $x^{-\frac{x+1}{2x}}$ is decreasing for every $x>0$. Hence if $2 < p < q$ then $l(D_{2p}) > l(D_{2q})$. Thus by Theorem \ref{main} we obtain that $G$ is $q$-nilpotent with $O_q(G)$ cyclic for every prime $q$. Therefore $G$ is cyclic. The second statement follows from the first statement and Lemma \ref{cyclic.few.divisors}.
\end{proof}

\section{Acknowledgements}
The authors are members of the ``National Group for Algebraic and Geometric Structures, and their Applications'' (GNSAGA - INdAM).

\bibliography{books.prod}

\begin{bibdiv}
\begin{biblist}

\bib{AJI}{article}{
      author={Amiri, H.},
      author={Jafarian~Amiri, S.~M.},
      author={Isaacs, I.~M.},
       title={Sums of element orders in finite groups},
        date={2009},
     journal={Comm. Algebra},
      volume={37},
       pages={2978\ndash 2980},
}

\bib{AKpnil}{article}{
      author={Baniasad~Azad, Morteza},
      author={Khosravi, Behrooz},
       title={A criterion for {$p$}-nilpotency and {$p$}-closedness by the sum
  of element orders},
        date={2020},
        ISSN={0092-7872},
     journal={Comm. Algebra},
      volume={48},
      number={12},
       pages={5391\ndash 5395},
         url={https://doi.org/10.1080/00927872.2020.1788571},
      review={\MR{4150745}},
}

\bib{AK}{article}{
      author={Baniasad~Azad, Morteza},
      author={Khosravi, Behrooz},
       title={Properties of finite groups determined by the product of their
  element orders},
        date={2021},
        ISSN={0004-9727},
     journal={Bull. Aust. Math. Soc.},
      volume={103},
      number={1},
       pages={88\ndash 95},
         url={https://doi.org/10.1017/S000497272000043X},
      review={\MR{4205763}},
}

\bib{BAKJ}{article}{
      author={Baniasad~Azad, Morteza},
      author={Khosravi, Behrooz},
      author={Jafarpour, Morteza},
       title={An answer to a conjecture on the sum of element orders},
        date={2022},
        ISSN={0219-4988},
     journal={J. Algebra Appl.},
      volume={21},
      number={4},
       pages={Paper No. 2250067, 8},
         url={https://doi.org/10.1142/S0219498822500670},
      review={\MR{4396137}},
}

\bib{BS}{article}{
      author={Beltr\'an, A.},
      author={S\'aez, A.},
       title={Existence of normal hall subgroups by means of orders of
  products},
        date={2019},
     journal={Math. Nachr.},
      volume={292},
       pages={720\ndash 723},
}

\bib{CGM}{article}{
      author={Contreras-Rojas, Y.},
      author={Grazian, V.},
      author={Monetta, C.},
       title={{$p$}-nilpotency criteria for some verbal subgroups},
        date={2022},
        ISSN={0021-8693},
     journal={J. Algebra},
      volume={609},
       pages={926\ndash 936},
         url={https://doi.org/10.1016/j.jalgebra.2022.07.017},
      review={\MR{4464532}},
}

\bib{DMN}{article}{
      author={Di~Domenico, E.},
      author={Monetta, C.},
      author={Noce, M.},
       title={Upper bounds for the product of element orders of finite groups},
        date={2023},
     journal={J. Algebr. Comb. (to appear)},
}

\bib{GP}{article}{
      author={Garonzi, M.},
      author={Patassini, M.},
       title={Inequalities detecting structural properties of a finite group},
        date={2016},
     journal={Comm. Algebra},
      volume={45},
       pages={677\ndash 687},
}

\bib{HML1}{article}{
      author={Herzog, M.},
      author={Longobardi, P.},
      author={Maj, M.},
       title={An exact upper bound for sums of elements order in non-cyclic
  finite groups},
        date={2018},
     journal={J. Pure Appl. Algebra},
      volume={222},
       pages={1628\ndash 1642},
}

\bib{HML3}{article}{
      author={Herzog, M.},
      author={Longobardi, P.},
      author={Maj, M.},
       title={Two new criteria for solvability of finite groups},
        date={2018},
     journal={J. Algebra},
       pages={215\ndash 226},
}

\bib{HML4}{article}{
      author={Herzog, M.},
      author={Longobardi, P.},
      author={Maj, M.},
       title={On a criterion for solvability of a finite groups},
        date={2021},
     journal={Comm. Algebra},
      volume={49},
       pages={2234\ndash 2240},
}

\bib{HML2}{article}{
      author={Herzog, M.},
      author={Longobardi, P.},
      author={Maj, M.},
       title={New criteria for solvability, nilpotency and other properties of
  finite groups in terms of the order elements or subgroups},
        date={2023},
     journal={Int. J. Group Theory},
}

\bib{ROB}{book}{
      author={Robinson, D.J.S.},
       title={A course in the theory of groups},
     edition={Second},
   publisher={Springer},
        date={1996},
}

\bib{Tarna1}{article}{
      author={T\u{a}rn\u{a}uceanu, M.},
       title={Detecting structural properties of finite groups by the sum of
  element orders},
        date={2020},
     journal={Isr. J. Math.},
      volume={238},
       pages={629\ndash 637},
}

\end{biblist}
\end{bibdiv}
\end{document}